\documentclass[12pt,leqno]{amsart}
\usepackage{amsmath,amsfonts,amssymb,amsthm}
\usepackage{graphicx}
\graphicspath{ }
\usepackage{wrapfig}
\usepackage{accents}
\usepackage{caption}
\usepackage{subcaption}
\usepackage{calligra}

\numberwithin{figure}{section}

\theoremstyle{plain}
\newtheorem{thm}{Theorem}[section]

\theoremstyle{definition}
\newtheorem{defn}{Definition}[section]

\numberwithin{equation}{section}
\theoremstyle{remark}

\usepackage{mathtools}
\title{Curves on a smooth surface with position vectors lie in the tangent plane}
\author[A. A. Shaikh and P. R. Ghosh]{Absos Ali Shaikh$^{*1}$ and Pinaki Ranjan Ghosh$^{2}$}
\address{\noindent\newline  $^1$Department of Mathematics,\newline University of
Burdwan, Golapbag,\newline Burdwan-713104,\newline West Bengal, India}
\email{aask2003@yahoo.co.in, aashaikh@math.buruniv.ac.in}
\address{\noindent\newline  $^2$Department of Mathematics,\newline University of
Burdwan, Golapbag,\newline Burdwan-713104,\newline West Bengal, India}
\email{mailtopinaki94@gmail.com}
\begin{document}

\begin{abstract}
The present paper deals with a study of curves on a smooth surface whose position vector always lies in the tangent plane of the surface and it is proved that such curves remain invariant under isometry of surfaces. It is also shown that length of the position vector, tangential component of the position vector and geodesic curvature of a curve on a surface whose position vector always lies in the tangent plane are invariant under isometry of surfaces.
\end{abstract}
\noindent\footnotetext{ $^*$ Corresponding author.\\
$\mathbf{2010}$\hspace{5pt}Mathematics\; Subject\; Classification: 53A04, 53A05, 53A15. }
\maketitle
\section{Introduction}
The notion of rectifying curve was introduced by Chen \cite{BYC03} as a curve in the Euclidean space such that its position vector always lies in the rectifying plane, and then investigated some properties of such curves. For further properties of rectifying curves, the reader can be consulted \cite{BYC05} and \cite{BYC18}. Again Ilarslan and Nesovic \cite{IN08} studied the rectifying curves in Minkowski space and obtained some of its characterization. 
\par
In \cite{CKI18} Camci et. al associated a frame different from Frenet frame to curves on a surface and deduced some characterization of its position vector. In \cite{PRG18A} and \cite{PRG18B} the present authors studied rectifying and osculating curves and obtained some conditions for the invariancy of such curves under isometry. Also the invariancy of the component of position vector of rectifying and osculating curves along the normal and tangent line to the surface are obtained under isometry of surfaces.
\par 
Motivating by the above studies of curves whose position vectors are confined in some plane, in this paper we have investigated curves on a smooth surface with position vector always lying in the tangent plane of the smooth surface. By using the Gauss equation we have deduced the component of the position vector along the tangent, normal and binormal vector in simple form. By considering isometry between two smooth surfaces it is proved that curves on smooth surface whose position vector lies in the tangent plane are invariant. It is also shown that the length of position vector, tangential component and geodesic curvature of such curves are invariant under isomerty.
\section{Preliminaries}
This section is concerned with some preliminary notions of rectifying curves, osculating curves, isometry of surfaces and geodesic curvature (for details see, \cite{AP01}, \cite{MPDC76 }) which will be needed for the remaining.
\par
At every point of an unit speed parametrized curve $\gamma(s)$ with atleast fourth order continuous derivative, there is an orthonormal frame of three vectors, namely, tangent, normal and binormal vectors. Tangent, normal and binormal vectors are denoted by $\vec{t}$, $\vec{n}$ and $\vec{b}$. They are related by the Serret-Frenet equation given as
\begin{eqnarray}
\nonumber
t'(s)&=&\kappa \ n(s),\\
\nonumber
n'(s) &=& -\kappa \ t(s) + \tau \ b(s),\\
\nonumber
b'(s)&=&-\tau \ n(s),
\end{eqnarray}
where $\kappa$ and $\tau$ are respectively the curvature and torsion of $\gamma(s)$. Rectifying, osculating and normal plane is generated by $\{\vec{t},\vec{b}\}$, $\{\vec{t},\vec{n}\}$ and $\{\vec{n},\vec{b}\}$ respectively. Curves whose position vector contained in rectifying, osculating and normal plane are respectively called rectifying, osculating and normal curves. 
\begin{defn}
Let $S$ and $\bar{S}$ be smooth surfaces immersed in $\mathbb{R}^3$. Then a diffeomorphism $f:S\rightarrow\bar{S}$ is called an isometry if the length of any curve on $S$ is invariant under $f$.
\end{defn}
\begin{defn}
Suppose $\gamma(s)$ is any unit speed parametrized curve on a smooth surface $S$. Then the tangent vector $\gamma'(s)$ and the normal $\vec{N}$ to the surface are mutually orthogonal and also $\gamma''(s)$ and $\gamma'(s)$ are orthogonal. Hence $\gamma''(s)$ is represented by the the linear combination of $\vec{N}\times\gamma'(s)$ and $\vec{N}$ as
\begin{equation*}
\gamma''(s)=\kappa_g\vec{N}\times\gamma'(s)+\kappa_n\vec{N}.
\end{equation*}
Then $\kappa_g$ and $\kappa_n$ are respectively called the geodesic curvature and normal curvature of $\gamma(s)$ on $S$ given by the following:
\begin{eqnarray}
\nonumber
\kappa_g&=&\gamma''\cdot(\vec{N}\times\gamma'),\\
\nonumber
\kappa_n&=&\gamma''\cdot\vec{N}.
\end{eqnarray}
\end{defn}
\section{Curves on a surface whose position vector lies in the tangent plane}
Let $S$ be a smooth surface and $\phi$ be a surface patch at any point $p\in S$. Let $\gamma(s)$ be an unit speed parametrized curve in $\phi$ passing through $p\in S$. Then the tangent space of $S$ at $p$ is generated by two linearly independent vectors $\phi_u$ and $\phi_v$, where $\phi_u=\frac{\partial \phi}{\partial u}$ and $\phi_v=\frac{\partial \phi}{\partial v}$. If $\gamma(s)$ be a curve on $S$ whose position vector lies in $T_{\gamma(s)}S$ then the equation of $\gamma(s)$ is given by
\begin{equation}\label{sc2}
\gamma(s)=\lambda(s)\phi_u+\mu(s)\phi_v,
\end{equation}
where $\lambda(s)$ and $\mu(s)$ are two functions of $s$.\\
Differentiating equation $(\ref{sc2})$ we get
\begin{equation*}
\gamma'(s)=\lambda'\phi_u+\mu'\phi_v+\lambda (u'\phi_{uu}+v'\phi_{uv})+\mu(u'\phi_{uv}+v'\phi_{vv}).
\end{equation*}
The Gauss Equation for the surface patch $\phi$ of $S$ with normal vector $\vec{N}$ is given by
\begin{equation}\label{sc3}
\begin{cases}
\phi_{uu}&=\Gamma_{11}^1\phi_u+\Gamma_{11}^2\phi_v+L\vec{N},\\
\phi_{uv}&=\Gamma_{12}^1\phi_u+\Gamma_{12}^2\phi_v+M\vec{N},\\
\phi_{vv}&=\Gamma_{22}^1\phi_u+\Gamma_{22}^2\phi_v+N\vec{N},\\
\end{cases}
\end{equation}
where $\{L, \ M, \ N\}$ are the coefficients of the second fundamental form of $S$, and the Christoffel symbols $\Gamma_{ij}^k$ are given by 
\begin{eqnarray}
\nonumber
\Gamma_{11}^1=\frac{E_uG-2F_uF+E_vF}{2(EG-F^2)},&& \ \Gamma_{11}^2=\frac{2F_uE-E_vE-E_uF}{2(EG-F^2)},\\
\nonumber
\Gamma_{12}^1=\frac{E_vG-G_uF}{2(EG-F^2)},&& \ \Gamma_{12}^2=\frac{G_uE-E_vF}{2(EG-F^2)},\\
\nonumber
\Gamma_{11}^1=\frac{2F_vG-G_uF-G_uG}{2(EG-F^2)},&& \ \Gamma_{11}^2=\frac{G_vE-2F_vF+G_uF}{2(EG-F^2)},
\end{eqnarray}
where $\{E, \ F, \ G\}$ are the coefficients of the second fundamental form of $S$.
Using equation $(\ref{sc3})$ in equation $(\ref{sc2})$ we get
\begin{eqnarray}
\nonumber
\vec{t}=\gamma'(s)=&&\{\lambda'+u'\lambda\Gamma_{11}^1+(v'\lambda+u'\mu) \Gamma_{12}^1+v'\mu\Gamma_{22}^1\}\phi_u+\{\mu'+u'\lambda\Gamma_{11}^2+(v'\lambda+u'\mu)\Gamma_{12}^2\\
\nonumber
&&+v'\mu\Gamma_{22}^2\}\phi_v+\{u'\lambda L+v'\lambda M+u'\mu M+v'\mu N\}\vec{N},
\end{eqnarray}
i.e., $$\vec{t}=A_1\phi_u+A_2\phi_V+A_3\vec{N},$$ where $A_1$, $A_2$ and $A_3$ are respectively given by 
\begin{equation}\label{sc4}
\begin{cases}
A_1=&\lambda'+u'\lambda\Gamma_{11}^1+(v'\lambda+u'\mu) \Gamma_{12}^1+v'\mu\Gamma_{22}^1,\\
A_2=&\mu'+u'\lambda\Gamma_{11}^2+(v'\lambda+u'\mu)\Gamma_{12}^2+v'\mu\Gamma_{22}^2,\\
A_3=& u'\lambda L+v'\lambda M+u'\mu M+v'\mu N.
\end{cases}
\end{equation}

\par
But the tangent plane is generated by $\phi_u$ and $\phi_v$, hence
\begin{eqnarray}
\nonumber
u'\lambda L+v'\lambda M+u'\mu M+v'\mu N=0,\\
\nonumber\frac{\lambda}{\mu}=-\frac{u'L+v'M}{u'M+v'N}.
\end{eqnarray}
So the tangent vector of $\gamma(s)$ is given by
\begin{equation*}
\vec{t}=A_1\phi_u+A_2\phi_v.
\end{equation*}
Here we consider that the curvature $\kappa$ of $\gamma$ is always positive. The normal vector $\vec{n}$ is given by
\begin{eqnarray}
\nonumber
\vec{n}&=&\frac{1}{\kappa}[A_1'\phi_u+A_2'\phi_v+A_1(\phi_{uu}u'+\phi_{uv}v')+A_2(\phi_{uv}u'+\phi_{vv}v')],\\
\nonumber
&=& \frac{1}{\kappa}\{B_1\phi_u+B_2\phi_v+B_3\vec{N}\},
\end{eqnarray}
where $B_1,\ B_2$ and $ B_3$ are respectively given by
\begin{equation}\label{sc5}
\begin{cases}
B_1&= A_1'+u'A_1\Gamma_{11}^1+(v'A_1+u'A_2)\Gamma_{12}^1+v'\Gamma_{22}^1A_2,\\
B_2&= A_2'+u'A_1\Gamma_{11}^2+(v'A_1+u'A_2)\Gamma_{12}^2+v'\Gamma_{22}^2A_2,\\
B_3&= u'A_1L+(v'A_1+u'A_2)M+v'A_2N.
\end{cases}
\end{equation}
Now the binormal vector $\vec{b}$ of $\gamma(s)$ is given by 
\begin{equation}\label{sc6}
\vec{b}=\frac{1}{\kappa}\{ (A_1B_2-A_2B_1)\vec{N}+A_1B_3(F\phi_u-E\phi_v)+A_2B_3(G\phi_u-F\phi_v)\}.
\end{equation}\\
\begin{thm}
Let $\gamma(s)$ be an unit speed parametrized curve on $S$ whose position vector lies in the tangent plane $T_{\gamma(s)}S$. Then the following statements hold:
\par
$(i)$ The distance function $\rho=\|\gamma\|$ is given by $\rho=\lambda^2E+2\lambda\mu F+\mu^2G$.\\

$(ii)$ The tangential component of the position vector of the curve $\gamma(s)$ is given by $$\langle \vec{t},\gamma \rangle=\lambda A_1E+(\lambda A_2+\mu A_1)F+\mu A_2G.$$ 

$(iii)$ The normal component of the position vector is given by $$\langle \vec{n},\gamma \rangle=\frac{1}{\kappa}(\lambda B_1E+(\lambda B_2+\mu B_1)F+\mu B_2G).$$

$(iv)$ The component of the position vector along binormal vector is given by $$\langle \vec{b},\gamma \rangle = \frac{1}{\kappa}\{A_1B_3\mu( F^2-EG)+A_2B_3\lambda( EG- F^2)\},$$
where $\{A_1,\ A_2,\ A_3\}$ and $\{B_1,\ B_2,\ B_3\}$ are described in equations $(\ref{sc4})$ and $(\ref{sc5})$ respectively.\\
\end{thm}
\begin{proof}
Let $\gamma(s)$ be an unit speed parametrized curve on $S$ whose position vector lies in the tangent plane $T_{\gamma(s)}S$ and curvature $\kappa>0$. Then $$\gamma(s)=\lambda\phi_u+\mu\phi_v.$$ Therefore
\begin{eqnarray}
\nonumber
\rho=\langle\gamma,\gamma\rangle&=&\langle\lambda\phi_u+\mu\phi_v,\lambda\phi_u+\mu\phi_v\rangle,\\
\nonumber
&=&\lambda^2\langle\phi_u,\phi_u\rangle+\lambda\mu\langle\phi_u,\phi_v\rangle+\lambda\mu\langle\phi_v,\phi_u\rangle+\mu\langle\phi_v,\phi_v\rangle,\\
\nonumber
&=&\lambda^2E+2\lambda\mu F+\mu^2G.
\end{eqnarray}
Which proves $(i)$.\par
The component of $\gamma$ along the tangent vector is obtained as
\begin{eqnarray}
\nonumber
\langle\vec{t},\gamma\rangle&=&\langle A_1\phi_u+A_2\phi_v,\lambda\phi_u+\mu\phi_v\rangle,\\
\nonumber
&=&\lambda A_1\langle\phi_u,\phi_u\rangle+\lambda A_2\langle\phi_u,\phi_v\rangle+\mu A_1\langle\phi_v,\phi_u\rangle+\mu A_2\langle\phi_v,\phi_v\rangle,\\
\nonumber
&=&\lambda A_1E+(\lambda A_2+\mu A_1)F+\mu A_2G,
\end{eqnarray}
where $A_1,\ A_2,\ A_3$ are given in equation $(\ref{sc4})$. This proves $(ii)$.\par
We also find the component of $\gamma$ along the normal vector, which is given by
\begin{eqnarray}
\nonumber
\langle\vec{n},\gamma\rangle&=&\langle \frac{1}{\kappa}\{B_1\phi_u+B_2\phi_v+B_3\vec{N},\lambda\phi_u+\mu\phi_v\rangle,\\
\nonumber
&=& \frac{1}{\kappa}\{\lambda B_1\langle\phi_u,\phi_u\rangle+\lambda B_2\langle\phi_u,\phi_v\rangle+\mu B_1\langle\phi_v,\phi_u\rangle+\mu B_2\langle\phi_v,\phi_v\rangle\},\\
\nonumber
&=&\frac{1}{\kappa}(\lambda B_1E+(\lambda B_2+\mu B_1)F+\mu B_2G),
\end{eqnarray}
where $B_1,\ B_2,\ B_3$ are given in equation $(\ref{sc5})$. Hence statement $(iii)$ is proved.\par
Again the component of $\gamma$ along the binormal vector is given by
\begin{eqnarray}
\nonumber
\langle\vec{b},\gamma\rangle&=&\langle \frac{1}{\kappa}\{(A_1B_2-A_2B_1)\vec{N}+A_1B_3(F\phi_u-E\phi_v)+A_2B_3(G\phi_u-F\phi_v)\},\lambda\phi_u+\mu\phi_v\rangle,\\
\nonumber
&=&\frac{1}{\kappa}\{A_1B_3(F\lambda\langle\phi_u,\phi_u\rangle+F\mu\langle\phi_u,\phi_v\rangle-E\lambda\langle\phi_u,\phi_v\rangle-E\mu\langle\phi_v,\phi_v\rangle)+A_2B_3(G\lambda\langle\phi_u,\phi_u\rangle\\
\nonumber
&&-F\lambda\langle\phi_v,\phi_u\rangle+G\mu\langle\phi_u,\phi_u\rangle-F\mu\langle\phi_v,\phi_v\rangle)\},\\
\nonumber
&=&\frac{1}{\kappa}\{A_1B_3\mu( F^2-EG)+A_2B_3\lambda( EG- F^2)\},
\end{eqnarray}
where $A_1,\ A_2,\ A_3,\ B_1,\ B_2,\ B_3$ are given in equation $(\ref{sc4})$ and $(\ref{sc5})$. Thus statement $(iv)$ is proved.
\end{proof}
Now suppose that $\bar{S}$ is an another surface isometric to $S$ and $f:S\rightarrow\bar{S}$ is the isometry. Then the curve $f\circ\gamma(s)$ in $\bar{S}$ is expressed as 
\begin{eqnarray}\label{sc7}
\nonumber
\bar{\gamma}(s)&=& f\circ\gamma(s)=f(\lambda\phi_u+\mu\phi_v),\\
\nonumber
&=&\lambda f_*(\phi_u)+\mu f_*(\phi_v),\\
&=&\lambda(\bar{\phi}_u)+\mu(\bar{\phi}_v),
\end{eqnarray}
for some smooth functions $\lambda$ and $\mu$. The position vector of $\bar{\gamma}$ lies in $T_{f(p)}\bar{S}$. Hence the isometry $f$ transforms a curve on a surface with position vector in the tangent plane to the same and $\lambda,\ \mu$ does not change under $f$. Hence
\begin{eqnarray}
\nonumber
\frac{u'L+v'M}{u'M+v'N}=\frac{u'\bar{L}+v'\bar{M}}{u'\bar{M}+v'\bar{N}},\\
\nonumber
u'^2(L\bar{M}-\bar{L}M)+v'^2(M\bar{N}-\bar{M}N)+u'v'(L\bar{N}-\bar{L}N)=0,
\end{eqnarray} 
where $\{L, \ M, \ N\}$ and $\{\bar{L}, \ \bar{M}, \ \bar{N}\}$ are the coefficients of the second fundamental form of $S$ and $\bar{S}$ respectively.
\begin{thm}
Let $f:S\rightarrow\bar{S}$ be an isometry and the position vector of two curves $\gamma(s)$ and $\bar{\gamma}(s)$ lies in $T_pS$ and $T_{f(p)\bar{S}}$ respectively. Then the following statements hold:
\par
$(i)$ The distance function $\rho=\|\gamma\|$ is invariant under the isometry. i.e., $\gamma(s)=\bar{\gamma}(s)$.

$(ii)$ The tangential component of the position vector of the curve $\gamma(s)$ is invariant under the isometry $f$. i.e., $\langle \vec{t},\gamma \rangle=\langle \vec{\bar{t}},\bar{\gamma} \rangle$. 

$(iii)$ 
%
%
The geodesic curvature of $\gamma$ is invariant under the isometry $f$.
\end{thm}
\begin{proof}
Let $f:S\rightarrow\bar{S}$ be an isometry. Since $\phi(u,v)$ is a surface patch for the surface $S$ at $p$, hence $\bar{\phi}(u,v)=f\circ\phi(u,v)$ is also a surface patch for $\bar{S}$   . Suppose $\{E,F,G\}$ and $\{\bar{E},\bar{F},\bar{G}\}$ are the coefficients of first fundamental forms of $\phi$ and $\bar{\phi}$ respectively. Then we have
\begin{equation}\label{i1}
\bar{E}=E,\quad \bar{F}=F\quad \text{and}\quad \bar{G}=G.
\end{equation} 
Differentiating first relation of equation $(\ref{i1})$ we get
\begin{equation}\label{i2}
\bar{E}_u=\frac{\partial}{\partial u}(\bar{\phi}_u\cdot\bar{\phi}_u)=\frac{\partial}{\partial u}(\phi_u\cdot\phi_u)=E_u.
\end{equation}
Similarly 
\begin{equation}\label{i3}
\begin{cases}
\bar{F}_u=F_u,\qquad \bar{G}_u=G_u,\qquad \bar{E}_v=E_v,\qquad\\ \bar{F}_v=F_v,\qquad and \qquad \bar{G}_v=G_v.
\end{cases}
\end{equation}\\
Now using $(\ref{i1})$ we have  
\begin{eqnarray}
\nonumber
\rho&=\lambda^2E+2\lambda\mu F+\mu^2G\\
\nonumber
&=\lambda^2\bar{E}+2\lambda\mu \bar{F}+\mu^2\bar{G}\\
\nonumber
&=\bar{\rho}.
\end{eqnarray} 
Hence the statement $(i)$ is proved.\par
Since $\bar{A_1}$ and $\bar{A_2}$ are functions of $\lambda$, $\mu$ and Christoffel symbols, hence by virtue of the equations $(\ref{i1}),\ (\ref{i2})$ and $(\ref{i3})$ we see that $\bar{A_1}$ and $\bar{A_2}$ are invariant under the isometry $f$ . Thus
\begin{eqnarray}
\nonumber
\langle \vec{\bar{t}},\bar{\gamma} \rangle&=&\bar{\lambda A_1}\bar{E}+(\lambda \bar{A_2}+\bar{\mu} \bar{A_1})\bar{F}+\bar{\mu} \bar{A_2}\bar{G},\\
\nonumber
&=& \lambda A_1E+(\lambda A_2+\mu A_1)F+\mu A_2G,\\
\nonumber
&=& \langle\vec{t},\gamma\rangle.
\end{eqnarray}
This proves $(ii)$. \par
Now 
\begin{eqnarray}
\nonumber
\kappa_g &=& \gamma''\cdot(\vec{N}\times\gamma'),\\
\nonumber
&=& \gamma''\{(\phi_u\times\phi_v)\times(A_1\phi_u+A_2\phi_v)\},\\
\nonumber
&=&\gamma''\cdot\{A_1(E\phi_v-F\phi_u)+A_2(F\phi_v-G\phi_u)\},\\
\nonumber
&=&(B_1\phi_u+B_2\phi_v+B_3\vec{N})\cdot\{A_1(E\phi_v-F\phi_u)+A_2(F\phi_v-G\phi_u)\},\\
\nonumber
&=&B_1A_1(EF-FE)+B_1A_2(F^2-EG)+B_2A_1(EG-F^2)+A_2B_2(FG-GF),\\
\nonumber
&=&(B_1A_2-B_2A_1)(F^2-GE).
\end{eqnarray}
Since $\bar{\gamma}$ is also a curve whose position vector lies in the tangent plane $T_{f(p)}\bar{S}$ of $\bar{S}$, hence $$\bar{\kappa}_g=(\bar{B}_1\bar{A}_2-\bar{B}_2\bar{A}_1)(\bar{F}^2-\bar{G}\bar{E}).$$ 
Since $\bar{B_1}$ and $\bar{B_2}$ are the functions of $\bar{A_1}$, $\bar{A_2}$, $\bar{A_1'}$, $\bar{A_2'}$ and Christoffel symbols, so by using the equations $(\ref{i1}),\ (\ref{i2})$ and $ (\ref{i3})$ we can say that $\bar{B_1}$ and $\bar{B_2}$ are invariant under the isometry $f$. In view of $(\ref{i1})$, the last equation yields $$\bar{\kappa}_g=\kappa_g,$$
which proves $(iii)$.
\end{proof}

\section{acknowledgment}
 The second author greatly acknowledges to The University Grants Commission, Government of India for the award of Junior Research Fellow.

\end{document}